\documentclass[12pt]{amsart}

\usepackage{graphicx}

\usepackage{comment}
\usepackage{amsfonts}
\usepackage{amssymb}
\usepackage{amsmath}
\usepackage{latexsym}
\usepackage{amscd}
\usepackage{lscape}
\usepackage{pifont}
\usepackage{fancyhdr}
\usepackage{amsthm}
\usepackage{amsrefs}
\usepackage[all,cmtip]{xy}
\newtheorem{thm}{Theorem}[section]

\newtheorem{lem}[thm]{Lemma}
\newtheorem{prop}[thm]{Proposition}
{\bf}{\it}
{\bf}{\it}
\newtheorem*{main-theorem}{Main Theorem}{\bf}{\it}
{\bf}{\it}
{\bf}{\it}
{\bf}{\it}
{\bf}{\it}
{\bf}{\it}
{\bf}{\it}
{\bf}{\it}
{\bf}{\it}
{\bf}{\it}
\newtheorem*{thurstonthm}{Thurston's Theorem}{\bf}{\it}

\newenvironment{pf*}[1]{\proof[#1]}{\endproof}
\usepackage{euscript}

\newcommand{\cal}[1]{{\mathcal #1}}

\newcommand{\beq}{\begin{equation}}
\newcommand{\eeq}{\end{equation}}

\newtheorem{defn}{Definition}[section]

\theoremstyle{remark}

\renewcommand{\mod}{\operatorname{mod}}
\newcommand{\tl}{\tilde}

\newcommand{\eps}{\epsilon}

\numberwithin{equation}{section}
\newcommand{\thmref}[1]{Theorem~\ref{#1}}
\newcommand{\propref}[1]{Proposition~\ref{#1}}

\newcommand{\lemref}[1]{Lemma~\ref{#1}}

\newcommand{\figref}[1]{Figure~\ref{#1}}

\newcommand{\cU}{{\cal U}}

\newcommand{\cM}{{\cal M}}

\newcommand{\cF}{{\cal F}}

\newcommand{\cT}{{\cal T}}

\newcommand{\cC}{{\cal C}}

\newcommand{\CC}{{\mathbb C}}

\newcommand{\NN}{{\mathbb N}}
\newcommand{\DD}{{\mathbb D}}
\newcommand{\HH}{{\mathbb H}}

\newcommand{\ignore}[1]{{}}

\begin{document}

\title[Geometrization of Thurston maps]{Geometrization of postcritically finite branched coverings}
\author{Sylvain Bonnot, Michael Yampolsky}
\date{\today}
\maketitle
\begin{abstract}
We study canonical decompositions of postcritically finite branched coverings of the 2-sphere, as defined by K.~Pilgrim. We show that every hyperbolic cycle in the decomposition does not have a Thurston obstruction. It is thus Thurston equivalent to a rational map.

\medskip
Nous \'etudions les d\'ecompositions canoniques de rev\^etements ramifi\'es de la sph\`{e}re, avec ensembles post-critiques finis, ainsi que K.~Pilgrim les a d\'efinies.
Nous montrons qu'aucun cycle hyperbolique dans la d\'ecomposition n'a d'obstruction de Thurston. Par cons\'equent, un tel cycle est \'equivalent au sens de Thurston \`{a} une application rationnelle.

\medskip
\noindent
MSC 37F20 (primary),37F30
\end{abstract}
\maketitle

\subsection*{Foreword} 
A proof of the main result was announced by Nikita Selinger at the Workshop
``Holomorphic Dynamics around Thurston's Theorem'' which took place at
Roskilde University
from September 27 - October 1, 2010 (a preprint \cite{sel} has soon appeared). At the same conference, we independently
proposed a different approach to the proof, which is presented here. We are very grateful to Nikita for several
 useful discussions, and, particularly, 
for pointing out an error in a previous version of this paper, and helping us to correct it in the proof. Our argument
bears a certain ideological similarity to that of \cite{sel}. However, it is based on a specific geometric surgery
construction, rather than the more abstract concept of augmented Teichm{\"u}ller space used by Nikita. 

\section{Introduction and statement of the result}

\paragraph*{\bf Thurston maps and multicurves}
In this section we recall the basic setting of Thurston's characterization of rational functions.
\label{section1.1}
Let $f:S^{2}\rightarrow S^{2}$ be an orientation-preserving branched covering map of the two-sphere. We define the \textit{postcritical set $P_{f}$} by
\[
 P_{f}:=\bigcup_{n>0}f^{\circ n}(\Omega_{f}),
\]
where $\Omega_{f}$ is the set of critical points of $f$. When the postcritical set $P_{f}$ is finite we say that $f$ is a \textit{Thurston mapping}.

Two Thurston maps $f$ and $g$ are \textit{Thurston equivalent} if there are homeomorphisms $\phi_{0}, \phi_{1}:S^{2}\rightarrow S^{2}$ such that
\begin{enumerate}
 \item the maps $\phi_{0}, \phi_{1}$ coincide on $P_{f}$, send $P_{f}$ to $P_{g}$ and are isotopic \text{rel } $P_{f}$; 
\item the diagram
\[
\begin{CD}
S^{2} @>\phi_{1}>> S^{2}\\
@VVfV @VVgV\\
S^{2} @>\phi_{0}>> S^{2}
\end{CD}
\]
commutes.
\end{enumerate}

Given a Thurston map $f:S^{2} \rightarrow S^{2}$, we define a function $N_{f}:S^{2} \rightarrow \mathbb{N}\cup {\infty}$ as follows:
\[
N_{f}(x)=\begin{cases}
1& \text{if $x \notin P_{f}$},\\
\infty & \text{if $x$ is in a cycle containing a critical point},\\
\underset{f^{k}(y)=x}{\text{lcm}} \text{deg}_{y}(f^{\circ k}) & \text{ otherwise}.
\end{cases}
\]

The pair $(S^{2},N_{f})$ is called the \textit{orbifold of $f$}. 
The {\it signature} of the orbifold $(S^2,N_f)$ is the set $\{N_f(x)\text{ for }x\text{ such that }1<N_f(x)<\infty\}$. 
The \textit{Euler characteristic} of the orbifold is given by 
\[
 \chi(S^{2},N_{f}):= 2-\sum_{x \in P_{f}}\left(1-\frac{1}{N_{f}(x)}\right).
\]
One can prove that $\chi(S^{2},N_{f})\leq 0$. In the case where $\chi(S^{2},N_{f})< 0$, we say that the orbifold is \textit{hyperbolic}. Observe that most orbifolds are hyperbolic: indeed, as soon as the cardinality $\vert P_{f} \vert >4$, the orbifold is hyperbolic.

\medskip
 We recall that a simple closed curve $\gamma \subset S^{2}-P_{f}$ is \textit{non essential} if it bounds a disk, and is \textit{peripheral} if it  bounds a punctured disk. We call a homotopy class of
simple closed curves $[\gamma]$ {\it trivial} if it is either non-essential or peripheral.

\begin{defn}A  \textit{multicurve} $\Gamma$ on $(S^{2},P_{f})$ is a set of disjoint, nonhomotopic, essential, nonperipheral simple closed curves on $S^{2}-P_{f}$.
A multicurve $\Gamma$ is \textit{f-stable} if for every curve $\gamma \in \Gamma$, each component $\alpha$ of $f^{-1}(\gamma)$ is either trivial or homotopic rel $P_{f}$ to an element of $\Gamma$. 

\end{defn}

To any $f$-stable multicurve is associated its \textit{Thurston linear transformation} $f_{\Gamma}:\mathbb{R}^{\Gamma}\rightarrow \mathbb{R}^{\Gamma}$, best described by the following transition matrix
\[
 M_{\gamma \delta}=\sum_{\alpha} \frac{1}{\text{deg}(f:\alpha \rightarrow \delta)}
\]
where the sum is taken over all the components $\alpha$ of $f^{-1}(\delta)$ which are isotopic rel $P_{f}$ to $\gamma$.
Since this matrix has nonnegative entries, it has a leading eigenvalue $\lambda(\Gamma)$ that is real and nonnegative (by the Perron-Frobenius theorem).

We can now state Thurston's theorem:

\begin{thurstonthm} Let $f:S^{2} \rightarrow S^{2}$ be a Thurston map
with hyperbolic orbifold. Then $f$ is Thurston equivalent to a rational function $g$ if and only if $\lambda(\Gamma)<1$ for every $f$-stable multicurve $\Gamma$. The rational function $g$ is unique up to conjugation with an automorphism of $\mathbb{P}^{1}$. 
\end{thurstonthm}

When a stable multicurve $\Gamma$ has a leading eigenvalue $\lambda(\Gamma)\geq 1$, we call it a \textit{Thurston obstruction}.

\paragraph*{\bf Pilgrim's canonical obstructions.}
Below we describe a particular type of Thurston obstructions, which were defined by K.~Pilgrim in \cite{Pil1}. 
Let us assume that a Thurston map $f$ has a hyperbolic orbifold.
Let us denote $\cT_f$ the Teichm{\"u}ller space of the 
punctured sphere $S\equiv S^2\setminus P_f$, and $d_{\cT_f}(\cdot,\cdot)$ the Teichm{\"u}ller distance; $\cM_f$ will denote the moduli
space of $S^2\setminus P_f$; $p_f:\cT_f\to \cM_f$ will be the covering map.
Further, for a choice of the complex structure $\tau$ on $S$, we let
$\rho_\tau$ denote the hyperbolic metric on the Riemann surface $S_\tau\equiv (S,\tau)$, 
$\text{length}_\tau$ the hyperbolic
length, and $d_\tau$ the hyperbolic distance. Similarly, for a general hyperbolic Riemann surface $W$ we denote
$\rho_W$, $d_W$, and $\text{length}_W$ the hyperbolic metric, distance, and length on $W$; $\cT_W$ the Teichm{\"u}ller space, etc.
\begin{defn}
For a non-trivial homotopy class of simple closed curves $[\gamma]$ on $S$ we let $\ell_\tau([\gamma])$ denote the length of the unique geodesic representative of $[\gamma]$ in $S_\tau$. 
\end{defn}
The map $f$ induces an analytic mapping on $\cT_f$:
$$\sigma_f:\cT_f\to\cT_f,\text{ where }\sigma_f([\tau])=[f^*\tau].$$
The map $\sigma_f$ does not increase Teichm{\"u}ller distance. 
Douady and Hubbard \cite{DH} show that the amount by which $\sigma_f$ contracts $d_{\cT_f}$ at a point $[\tau]\in\cT_f$
depends only on $p_f(\tau)$ and a {\it finite} amount of additional
information. More specifically: 
\begin{prop}[Lemma 5.2 of \cite{DH}]
\label{finite cover}
There exists a tower $$\cT_f\overset{\tl p_f}{\longrightarrow}\widetilde \cM_f\overset{\bar p_f}{\longrightarrow}\cM_f$$ of covering spaces,
such that $\bar p_f$ is a finite cover, and a map $\tilde \sigma_f:\widetilde \cM_f\to\cM_f$, such that the diagram below commutes:
\[
\begin{CD}
\cT_f @>{\sigma_f}>> \cT_f    \\
@VV{\tl p_f}V @VV{p_f}V\\
\widetilde\cM_f@>{\tl\sigma_f}>> \cM_f
\end{CD}
\]
The Teichm{\"u}ller norm $||D_{[\tau]}\sigma_f||$ of the differential of $\sigma_f$ depends only on the projection
$\tl p_f([\tau])$. 
\end{prop}

There exists a rational mapping $R:\hat\CC\to\hat\CC$ which is Thurston equivalent to $f$ if and only if there is 
a fixed point $[\tau_*]=\sigma_f[\tau_*]$. 
In the absence of a Thurston obstruction, since $\sigma^2_f$ is strictly contracting \cite{DH}, 
for any choice of the starting point $[\tau_0]\in\cT_f$, 
the iterates $[\tau_n]\equiv \sigma_f^n([\tau_0])$ converge to $[\tau_*]$ geometrically fast.

\propref{finite cover} and contracting properties of $\sigma_f$ easily imply:
\begin{prop}[Proposition 5.1 of \cite{DH}]
\label{over compact}
The iterates $[\tau_n]\equiv \sigma_f^n([\tau_0])$ converge in $\cT_f$ to $[\tau_*]$ which is a fixed point of $\sigma_f$
if and only if the sequence $\{p_f[\tau_n]\}$ is pre-compact in $\cM_f$.
\end{prop}

Pilgrim showed that a presence of an obstruction implies that the sequence $[\tau_n]$ diverges to infinity in $\cT_f$ in the following specific sense:
\begin{thm}[\cite{Pil1}]
\label{canonical1}
Suppose $f$ is obstructed. Then
\begin{itemize}
\item[(I)] there exists a class $[\gamma]$ such that
$$\ell_{\tau_n}([\gamma])\to 0;$$
\item[(II)] for a non-trivial homotopy 
class $[\gamma]$ the above property is independent of the starting point $[\tau_0]\in\cT_f$;
\item[(III)] the union of all classes $[\gamma]$ as above forms a Thurston obstruction $\Gamma_c$.
\end{itemize}
\end{thm} 

\noindent
Pilgrim calls $\Gamma_c$ the {\it canonical} Thurston obstruction. Thus, the existence of an obstruction implies that the canonical obstruction exists (that is $\Gamma_c\neq\emptyset$).

Pilgrim further showed:
\begin{thm}[\cite{Pil1}]
\label{pil-bound}
Let $[\tau_0]\in\cT_f$. There exists a constant $E=E([\tau_0])$ such that for every
non-trivial simple closed curve $\gamma\notin \Gamma_c$ we have
$$\inf \ell_{\sigma^n_f\tau_0}([\gamma])>E.$$

\end{thm}


\paragraph*{\bf Pilgrim's decompositions and combinations of Thurston maps}
What follows is a very brief review; the reader is referred to K.~Pilgrim's book \cite{Pil2} for details.
We adhere to the notation of \cite{Pil2}, for ease of reference.

As a motivation, consider that for the canonical Thurston obstruction $\Gamma_c\ni\gamma$, 
there is a choice of complex structure $\tau$
for which $\ell_\tau([\gamma])$ is arbitrarily small, and remains small under pull-back by $f$. It is thus natural
to think of the punctured sphere $S^2\setminus P_f$ as pinching along the homotopy classes $[\gamma]\in\Gamma_c$;
instead of a single sphere we then obtain a collection of spheres, interchanged by a map $f$. 

More specifically,
let $f$ be a Thurston map, and $\Gamma=\cup\gamma_j$ an $f$-stable multicurve. Consider also a finite collection of 
disjoint annuli
$A_{0,j}$ whose core curves are the respective $\gamma_j$. For each $A_{0,j}$ consider only non-trivial preimages; these form
a collection of annuli $A_{1,k}$ each of which is homotopic to one of the curves in $\Gamma$.
Pilgrim says that the pair $(f,\Gamma)$
 is in a {\it standard form} (see Figure \ref{fig-decomp})
if there exists a collection of annuli $A_{0,j}$ as above such that the
following properties hold:
\begin{itemize}
\item[(a)] for each curve $\gamma_j$ the annuli $A_{1,k}$ in the same homotopy class are contained inside $A_{0,j}$;
\item[(b)] moreover, the two outermost annuli $A_{1,k}$ as above share their outer boundary curves with $A_{0,j}$;
\item[(c)] finally, restricted to a boundary curve $\chi$ of $A_{0,j}$, the map $f:\chi\to f(\chi)$ is, up to a homeomorphic
change of coordinates in the domain and the range, given by  $z\mapsto z^d:S^1\to S^1$, for some $d\geq 1$.
\end{itemize}

\begin{figure}[ht]
\includegraphics[width=\textwidth]{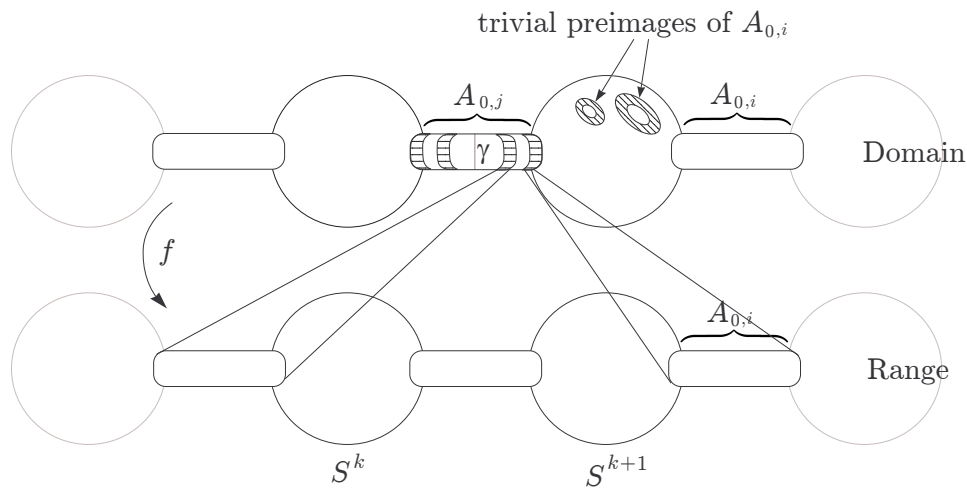}
\caption{\label{fig-decomp}Pilgrim's decomposition of a Thurston mapping}
\end{figure}

A Thurston map with a multicurve 
in a standard form can be decomposed as follows. First, all annuli $A_{0,j}$ are removed, leaving a collection of spheres with holes, denoted $S_0(j)$. For each $j$, there exists a unique connected component $S_1(j)$ of
$f^{-1}(\cup S_0(j))$ which has the property $\partial S_0(j)\subset \partial S_1(j)$. Any such $S_1(j)$ is a sphere with
holes, with boundary curves being of two types: boundaries of removed annuli, or boundaries of trivial preimages of the
removed annuli. 

The holes in $S_0(j)\subset S^2$ can be filled using the property (c) above. Namely, let $\chi$ be a boundary curve of
a component $D$ of $S^2\setminus S_0(j)$. 
Let $k\in\NN$ be the first return $f^k:\chi\to\chi$, if it exists. For each $0\leq i\leq k-1$ the curve 
$\chi_i\equiv f^i(\chi)$ bounds a component $D_i$ of $S^2\setminus S_0(m_i)$ for some $m_i$. Denote $d_i$ the degree
of $f:\chi_i\to \chi_{i+1}$.  Select  homeomorphisms 
$$h_i:\bar D_i\to \bar \DD\text{ so that }h_{i+1}\circ f\circ h_i^{-1}(z)=z^{d_i}.$$
Set $\tl f\equiv f$ on $\cup S_0(j)$.
Define new punctured spheres $\tl S(j)$ by adjoining punctured caps $D_i^*\equiv h_i^{-1}(\bar \DD\setminus \{0\})$
to $S_0(j)$. Extend the map $\tl f$ to each $D^*_i$ by setting 
$$\tl f(z)=h_{i+1}^{-1}\circ (h_i(z))^d.$$
We have thus replaced every hole with a cap with a single puncture.

By construction, the map 
$$\tl f:\cup \tl S(j)\to\cup\tl S(j)$$
contains a finite number of periodic cycles of  punctured spheres. For every periodic cycle of spheres, pick a representative
 $\tl S(j)$, and denote
by $\cF$ the first return map $f^{k_j}:\tl S(j)\to\tl S(j).$ This is again a Thurston map.
The collection of maps $\cF$ and the combinatorial information required to glue the spheres $S_0(j)$ back together 
is what Pilgrim calls a {\it decomposition} of $f$.

Pilgrim shows:
\begin{thm}
For every obstructed Thurston map $f$ there exists an equivalent map $g$, whose canonical obstruction
we denote $\Gamma_c^g$, such that $(g,\Gamma_c^g)$ is in a standard form, and thus can be decomposed.
\end{thm}

\paragraph*{\bf Statement of the geometrization result.}
It is natural to ask whether the canonical decomposition described above has the maximality property: that is,
whether the restrictions of the return map $\cF$ to spheres $\tl S(j)$, which have topological degrees greater
than one, and a hyperbolic orbifold, are {\it unobstructed}. 
In view of Thurston's theorem, this would imply that for every such $\tl S(j)$ there is a unique, up to a normalization, rational map $R:\hat\CC\to\hat\CC$ which is equivalent to $\cF:\tl S(j)\to\tl S(j)$. 
This constitutes the main conjecture posed by Pilgrim in \cite{Pil2}.

In our main result we prove Pilgrim's conjecture. Thus, Pilgrim's decomposition of an obstructed Thurston mapping canonically breaks it into unobstructed, and thus {\it geometrizable}, pieces.

\begin{main-theorem}
Let $\cF:\cup\tl S(j)\to\cup\tl S(j)$ be given by Pilgrim decomposition of an obstructed Thurston map along the
canonical obstruction. For every $j$ such that $\cF:\tl S(j)\to\tl S(j)$ has a topological degree greater than $1$
and a hyperbolic orbifold, there does not exist a Thurston obstruction in $\tl S(j)$.
\end{main-theorem}

\section{Preliminaries}
\label{section-generalities}

\subsection*{Some notation}
Let $\tl S(j)$ be as in the statement of the main theorem. It is obtained by adding caps $D_i$ to a sphere with
holes $S_1(j)\subset S$. The restriction of the first return map $\cF$ to $\tl S(j)$ is an iterate $\tl f^N$,
so that we have a cycle of punctured spheres
\begin{equation}
\label{cycle-spheres}
S^0\equiv \tl S(j)\overset{\tl f}{\longrightarrow}S^1\overset{\tl f}{\longrightarrow}\cdots\overset{\tl f}{\longrightarrow}S^{N-1}\overset{\tl f}{\longrightarrow}S^0.
\end{equation}

\subsection*{Collar Lemmas}
Let us denote by $s(x)$ the function
$$s(x)=\sinh^{-1}(1/\sinh(x/2)).$$
Note that $s(x)$ decreases from $\infty$ to $0$ as $x$ increases from $0$ to $\infty$.
The {\it collar} around a simple closed hyperbolic geodesic $\gamma$ on a hyperbolic Riemann surface $W$
is the neighborhood
$$\cC(\gamma)\equiv \{z\in W\;|d_W(z,\gamma)<s(\text{length}_W(\gamma))\}.$$
The following is known as  Collar Lemma (cf. \cite{Buser}):

\begin{thm}[{\bf Collar Lemma for closed geodesics}]
The collar $\cC(\gamma)$ is an annulus. Further, if $\gamma$ and $\delta$ are two disjoint simple closed
geodesics on a hyperbolic Riemann surface $W$, then
$$\cC(\gamma)\cap\cC(\delta)=\emptyset.$$
\end{thm}

\noindent
We also recall  a limiting version of Collar Lemma for cusps (see \cite{McM}).

\begin{lem}[{\bf Collar Lemma for cusps}]
Let us denote $\cC$ the quotient of the region
\[
\{z : \operatorname{Im}(z) >1   \} \subset \mathbb{H}
\]
by the translation $z \mapsto z+2$. 
For every cusp $\kappa$ of a hyperbolic 
surface $X$ there is an isometry $$\iota_\kappa:\cC\to C(\kappa)$$ 
between $\cC$ and  a collar neighborhood $C(\kappa) \subset X$.
Furthermore, the collars about different cusps are disjoint, and $C(\kappa)$ is disjoint from the 
collar $C(\gamma)$ about any simple geodesic $\gamma$ on $X$.

\end{lem}

For $\Delta>1$ we denote $\cC_\Delta\subset \cC$ the quotient of $\{z : \text{Im}(z) >\Delta   \} \subset \mathbb{H}$ 
by $z\mapsto z+2$, and call 
$$C_\Delta(\kappa)\equiv \iota_\kappa(\cC_\Delta)$$ the {\it $\Delta$-neighborhood of the cusp $\kappa$.}

\ignore{

\subsection*{Gluing and decomposing the spheres}

Pilgrim's {\it combination} applied to
$$\tl f:\cup\tl S(j)\to\cup\tl S(j)$$
involves connecting the punctured spheres by annuli to reverse the decomposition construction. The end result
is a branched covering which is equivalent to $f:S^2\to S^2$. 
We need to fix some of the geometric 
properties of the decomposition construction, for future use. 
Namely, let $S_1\equiv \tl S(j_1)$ and 
$S_2\equiv \tl S(j_2)$ are two punctured spheres to be combined at the pair of punctures
$p_1\in S_1$ and $p_2\in S_2$.
We will glue in an annulus by removing a small topological disk $D_i\subset S_i$ around each of the punctures
$p_i$, and attaching the boundary curves of the annulus to the boundary circles $\partial D_i$. Since there
are no closed geodesics surrounding a cusp, we will instead use geodesic bigons as described below.
Let $S_1, \ldots , S_N$ be a finite collection of punctured spheres, each one being a 
hyperbolic Riemann surface. For each puncture $p_{j} \in S_1 \sqcup \ldots \sqcup S_N$ we
choose a disk around it $D_{p_j}$ that contains only the puncture $p_j$. 

\begin{defn}[Admissible pairing of caps]
We say that a finite family of pairs $({p_i},{p_j})$ is {\it an admissible pairing of punctures} if
\begin{enumerate}
\item for each pair $({p_i},{p_j})$ the punctures $p_i$ and $p_j$ belong to distinct spheres;
\item the connected sum $S_1 \# \cdots \# S_N$ obtained by gluing along the pairs $({p_i},{p_j})$ is a punctured sphere with at least
three punctures.
\end{enumerate} 
\end{defn}
%


\noindent
We can now describe combining $S_1$ and $S_2$.
Note, that the gluing of two hyperbolic pieces along (piecewise) geodesic boundaries is possible whenever the angles at the corners of the boundaries add up to $2 \pi$ and the lengths of the boundaries agree. We state the following:

\begin{lem}
Let us consider two hyperbolic punctured spheres $S_1, S_2$ each one with a distinguished puncture $p_1 \in S_1$ and $p_2 \in S_2$, as above. 
There exists an $\epsilon_{0}>0$ such that for any $\epsilon \leq \epsilon_0$ and every $\Delta>1$ the following is true:
\begin{enumerate}
\item for an angle 
$\alpha=\alpha(\epsilon,\Delta)$
there exist topological disks $D_1\ni p_1$, $D_2\ni p_2$ in the $\Delta$-neighborhoods of 
$p_1$, $p_2$ respectively,
which are isotopic to the collar neighborhoods of the cusps (relative to the complements of the collar neighborhoods), 
and such that the boundary circle $C_i\equiv\partial D_i$ is 
made of two geodesic arcs meeting at two points with the angle $\alpha$;
\item there exists a hyperbolic annulus $A$ from such that the core curve of $A$ has length $\epsilon$, and  $A$
is bounded by two topological circles $C_1'$ and $C_2'$, each one of which is made of two geodesic arcs meeting at two points at angle $2 \pi - \alpha$, and having the same lengths as the geodesics forming $C_1$ and $C_2$ respectively.
\end{enumerate}

\end{lem}
\begin{proof}
Using Gauss-Bonnet Theorem on topological annuli bounded by piecewise geodesic arcs, one can find arbitrarily thin annuli bounded by pairs of geodesic arcs meeting at an angle arbitrarily close to $\pi$. Using the same theorem, one can find circles made of two geodesic arcs meeting at an angle arbitrarily close to $2 \pi - \pi = \pi$ near a given cusp.

\end{proof}

By applying the above statement repeatedly, we obtain the following lemma.

\begin{lem}
\label{gluing1}
Let $S_1, \ldots , S_{N}$ be a collection of hyperbolic punctured spheres together with an admissible pairing of punctures 
$\{({p_{i_1}},{p_{j_1}}), \ldots  ({p_{i_K}},{p_{j_K}}) \}$. There exists $\epsilon_{0}>0$ so that for any $\epsilon<\epsilon_{0}$ and $\Delta>1$
one can perform a gluing on all the pairs $({p_{i}},{p_{j}})$ by using hyperbolic annuli $A_{p_i,p_j}$ so that:
\begin{enumerate} 
\item the hyperbolic metric on each $S_j$ is isometric to the hyperbolic metric on $S_1\#\cdots\# S_N$;
\item the boundary circles of the annuli are geodesic bigons contained in the $\Delta$-neighborhoods of the cusps;
\item the length of the core curve of each $A_{p_i,p_j}$ is equal to $\epsilon$. 
\end{enumerate}
\end{lem}

}

\subsection*{Teichm{\"u}ller spaces}
Here we describe a key technical tool that we will use, which is due to Y. Minsky (see \cite{Minsky}).

Let $S$ be a surface of finite type, and $\Gamma = \{\gamma_{1}, \ldots, \gamma_{k}\}$ be a system of disjoint, 
homotopically distinct, non trivial simple closed curves on $S$. 
Then we denote by $\text{Thin}_{\epsilon}(S,\Gamma)$ the {\it thin} part of the Teichm{\"u}ller space of $S$ 
given by
$$\text{Thin}_\epsilon(S,\Gamma)=\{[\tau]\in\cT_s\text{ for which }\underset{i}{\sup}\;\ell_\tau([\gamma_i])<\epsilon\}.$$
Denote  $S_{\Gamma}$  the multiply connected surface obtained from $S$ by decomposing along the collection $\gamma_i$
(that is, cutting along $\gamma_i$, and capping every hole by a puncturd disk).
There is a natural homeomorphism 
\[
\Pi_\Gamma: \cT_S \rightarrow \cT_{S_{\Gamma}}\times \mathbb{H}_1 \times \ldots \times \mathbb{H}_k
\]
defined as follows. Firstly, we complete the family $\Gamma$ into a larger family of curves $\hat\Gamma$ decomposing $S$ into pairs of 
pants (possibly degenerate). 
Recall the definition of Fenchen-Nielsen coordinates 
$$(\ell_\tau(\gamma),\text{Twist}_\tau(\gamma))_{\gamma\in\hat\Gamma}\text{ on }\cT_S,$$
given by the homotopy classes of curves in $\hat\Gamma$. 
The projection 
$\cT_S \mapsto \cT_{S_{\Gamma}}$ is obtained by restriction of these coordinates to $\hat\Gamma\setminus\Gamma$.
For each $\gamma_i\in\Gamma$, the map $\Pi_\Gamma$ sends
$$(\ell_\tau([\gamma_i]), \text{Twist}_\tau([\gamma_i]))\longrightarrow
\left( \frac{1}{\ell_\tau([\gamma_i])}, \text{Twist}_\tau([\gamma_i])\right) \in \mathbb{H}_i.$$

Minsky's result gives a distortion bound on this homeomorphism for small enough values of $\epsilon$:
\begin{thm}[\cite{Minsky}]
\label{minsky thm}
Endow the product $\cT_{S_{\Gamma}}\times \mathbb{H}_1 \times \ldots \times \mathbb{H}_k$ with the 
maximum of
the Teichm{\"u}ller distances on
the connected components of $\cT_{S_\gamma}$, and the 
hyperbolic distances on each copy of $\HH$.
Then, for all $\epsilon>0$ sufficiently small,
the  homeomorphism $\Pi_\Gamma$, restricted to $\operatorname{Thin}_{\epsilon}(S, \gamma)$ distorts distances by a bounded additive amount.
\end{thm}

\subsection*{Decompositions}

To fix the ideas, for the remainder of this paper,
 the decomposition of a Riemann surface $S$ with a canonical multicurve $\Gamma_c$
on the level of Teichm{\"u}ller spaces
will be carried out by restricting the Fenchen-Nielsen projection
\begin{equation}
\label{fnp}
\Pi_\Gamma: \cT_S \rightarrow \cT_{S_{\Gamma_c}}\times \mathbb{H}_1 \times \ldots \times \mathbb{H}_k
\end{equation}
to the first coordinate, by post-composing with the forgetful map
\begin{equation}
\label{frgt}
\cT_{S_{\Gamma_c}}\times \mathbb{H}_1 \times \ldots \times \mathbb{H}_k\longrightarrow\cT_{S_{\Gamma_c}}.
\end{equation}

\ignore{
Geometrically, this is achieved by cutting and capping 
the Riemann surface $S$ {\em preserving the hyperbolic structure. }
Note, that the gluing of two hyperbolic pieces along (piecewise) geodesic boundaries
 is possible whenever the angles at the corners of the boundaries add up to $2 \pi$ and 
the lengths of the boundaries agree. We cannot decompose $S$ along geodesic loops in $\Gamma_c$, as there
are no closed geodesics surrounding a cusp. Instead, we will use {\it geodesic bigons} bounded by geodesic arcs
meeting at the same angles $\alpha$ at the two vertices. An easy application of Gauss-Bonnet Theorem 
together with the Collar Lemmas shows that:
\begin{prop}
\label{caps}
There exists $\delta_0>0$ such that, assuming that
$$\max_{\gamma\in\Gamma}\ell_\tau([\gamma])<\delta_0,$$ 
the surface
$(S,\tau)$ can be decomposed into $\cup\tl S(j)$ in such a way that:
\begin{itemize}
\item  the inclusion of the surface with holes $S_1(j)\subset \tl S(j)$ into $(S,\tau)$
is a hyperbolic isometry;
\item the caps $D_i$ are bounded by geodesic bigons.
\end{itemize}
Furthermore, for every $\eps>0$ there exists $\delta>0$ such that if $$\max_{\gamma\in\Gamma}\ell_\tau([\gamma])<\delta,$$
then the surface can be decomposed so that every cap $D_i\subset S_0(j)\setminus S_1(j)$ has spherical diameter
less than $\eps$ in $(S,\tau)$.
\end{prop}

}

\subsection*{Controlling hyperbolic lengths}
The following facts will help us in bounding the changes to hyperbolic lengths, and complex structures, when 
decomposing Riemann surfaces. First, we
note a standard consequence of the comparison of the hyperbolic metric on a connected 
subdomain $\Omega$
 of $\hat\CC$ with $\frac{1}{d(z,\partial \Omega)}\vert dz \vert$ (see e.g. \cite{Minda}):
\begin{lem}
\label{hypest1}
Let $\Omega$  denote the Riemann sphere $\hat\CC$  with $p$ punctures ($p\geq 3$) with its Poincar\'e metric $\rho_\Omega$,
 and let $\gamma$ be a non peripheral curve. 
For any constant $\eps>0$ there exists $\delta>0$ so that the following holds. 
Let $D_i$ be a collection of disks around the punctures with spherical radii at most $\delta$. Then 
$$\text{length}_{\Omega\setminus(\cup D_i)}(\gamma)<\text{length}_\Omega(\gamma)+\eps.$$
\end{lem}

\noindent
Furthermore, we have:
\begin{lem}
\label{hypest3}
For every $\eps>0$ there exists $M>0$ such that the following holds. Suppose $S$ is the Riemann sphere with $p\geq 3$ punctures,
$A\subset S$ is an annulus such that $\mod(A)>M$, and $S_1$ and $S_2$ are the two connected components of $S\setminus A$.
 Then for every pair of points 
$a,b$ in $S_1$ we have
$$d_{S\setminus S_2}(a,b)<d_S(a,b)+\eps.$$

\end{lem}

Finally, we state (cf. \cite{Hai}):

\begin{lem}
\label{haiss}
Let $X$ be a finite type hyperbolic surface and denote $\sigma_0$ the standard complex structure in $X$.
\begin{itemize}
\item [(I)] Let $x\in X$, and $\cU\subset \cT_X$ any neighborhood of $[\sigma_0]$ in the Teichm{\"u}ller space of $X$.
Denote $B_r(x)$ the hyperbolic ball with radius $r>0$ around $X$. 
There exists $r_0>0$ such that for any complex structure $\sigma$ on $X$ with $\sigma\equiv \sigma_0$ outside $B_r(x)$
we have $[\sigma]\in \cU$.
\item [(II)] Let $\kappa$ be a cusp in $X$, and $\cU\subset \cT_X$ any neighborhood of $[\sigma_0]$ in the Teichm{\"u}ller space of $X$.
There exists a cusp neighborhood $C_\Delta (\kappa)$ such that
for any complex structure $\sigma$ on $X$ with $\sigma\equiv \sigma_0$ outside $C_\Delta(\kappa)$
we have $[\sigma]\in \cU$.
\end{itemize}
\end{lem}

\subsection*{Combinations of punctured surfaces}
Let us briefly describe the inverse of the above decomposition procedure, for future reference. Pilgrim's {\it combination} applied to
$\cup\tl S(j)$
involves connecting the punctured spheres by annuli to reverse the decomposition construction. The end result
is a branched covering which is equivalent to $f:S^2\to S^2$. 
Let $S_1\equiv \tl S(j_1)$ and 
$S_2\equiv \tl S(j_2)$ be two punctured spheres to be combined at the pair of punctures
$p_1\in S_1$ and $p_2\in S_2$.
We will glue in an annulus by removing a small topological disk $D_i\subset S_i$ around each of the punctures
$p_i$, and attaching the boundary curves of the annulus to the boundary circles $\partial D_i$. 
We again would like to perform the gluing in a hyperbolically rigid fashion.

Let $W_1, \ldots , W_N$ be a finite collection of punctured spheres, each one being a 
hyperbolic Riemann surface. For each puncture $p_{j} \in W_1 \sqcup \ldots \sqcup W_N$ we
choose a disk around it $D_{p_j}$ that contains only the puncture $p_j$. 

\begin{defn}[Admissible pairing of caps]
We say that a finite family of pairs $({p_i},{p_j})$ is {\it an admissible pairing of punctures} if
\begin{enumerate}
\item for each pair $({p_i},{p_j})$ the punctures $p_i$ and $p_j$ belong to distinct spheres;
\item the connected sum $W_1 \# \cdots \# W_N$ obtained by gluing along the pairs $({p_i},{p_j})$ is a punctured sphere with at least
three punctures.
\end{enumerate} 
\end{defn}

\noindent
We can now describe combining $S_1$ and $S_2$, using \lemref{hypest3}:
\begin{lem}
\label{compn}
Let us consider two hyperbolic punctured spheres $S_1, S_2$ each one with a distinguished puncture $p_1 \in S_1$ and $p_2 \in S_2$. For every $\eps>0$ there exists $\Delta>1$ such that the following holds.
Let $D_1\ni p_1$, $D_2\ni p_2$ be two topological disks  in the respective $\Delta$-neighborhoods of 
$p_1$, $p_2$ respectively. 
Let $S$ be a combination of the surfaces $S_1'\equiv S_1\setminus D_1$ and $S_2'\equiv S_2\setminus D_2$. That is,
there exist conformal embeddings 
$$\chi_i:S_i'\to S$$
with disjoint images, such that
$$A\equiv S\setminus (\chi_1(S_1')\cup \chi_2(S_2'))$$
is an annulus with a core curve $\gamma$. 
Denote $\hat S_1$, $\hat S_2$ the decomposition of $S$ along $\Gamma=\{\gamma\}$.
Then:
\begin{itemize}
\item $\text{length}_S(\gamma)<\eps,$ and 
\item $d_{\cT_{S_i}}(\hat S_i,S_i)<\eps.$

\end{itemize}

\end{lem}

By applying the above statement repeatedly, we obtain the following:

\begin{lem}
\label{gluing1}
Let $S_1, \ldots , S_{N}$ be a collection of hyperbolic punctured spheres together with an admissible pairing of punctures 
$\{({p_{i_1}},{p_{j_1}}), \ldots  ({p_{i_K}},{p_{j_K}}) \}$.
For every $\eps>0$ there exists $\Delta>1$ such that the following holds.
Let $D_i\subset C_\Delta(p_i)$ be a topological disk around $p_i$. Let 
$$S_i'\equiv S_i\setminus \cup(D_j).$$
Then the surfaces $S_i$ can be combined into  a hyperbolic surface $S$ according to the above pairing of punctures so that:
\begin{itemize}
\item each $S_i'$ is  conformally embedded into $S$ with complementary annuli $A_{p_i,p_j}$;
\item the length of the core curve $\gamma_{i,j}$ of each $A_{p_i,p_j}$ is less than $\epsilon$;
\item denoting $\cup \hat S_i$ the decomposition of $S$ along $\Gamma=\cup \gamma_{i,j}$,
we have
$$d_{\cT_{S_i}}(S_i,\hat S_i)<\eps\text{ for all }i.$$
\end{itemize}
\end{lem}

\begin{figure}[ht]
\includegraphics[width=\textwidth]{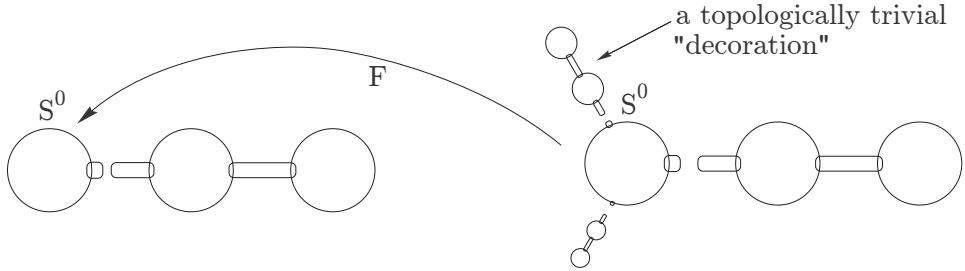}
\caption{\label{fig-decor}Pull-back of a decomposition}
\end{figure}

\paragraph*{\bf Idea of the proof of the Main Theorem.} 
When we pull-back a complex structure $\tau_0$ by $f$, only the geodesics of the canonical restriction become arbitrarily short. If we
restrict to a subsurface $S_1(j)$, then, with respect to the structure  
$$\tau_n\equiv \sigma^n_f\tau_0,$$ all simple closed curves have hyperbolic length
bounded from below independently of $n$ (\thmref{pil-bound}). However, if we restrict $\tau_0$ to a complex structure $\mu_0$ on 
$\tl S_0(j)$, and pull back by $\sigma_\cF$, the result may {\it a priori} be different. As illustrated in \figref{fig-decor}, 
the restrictions of $\tau_n$ to $S_0(j)$ will include ``decorations'': trivial preimages of the other subsurfaces. They are attached
by very thin tubes, and will not create much difference for the first few pull-backs. However, without some uniform distortion bound on the
change in the complex structure caused by restriction to the subsurface,
 we will not be able to control their impact on the length
of geodesics. 
Here Minsky's Theorem \ref{minsky thm} will be a key tool for us, as it produces a uniform bound on the change in complex structure 
caused by a decomposition.

\section{Proof of the Main Theorem}

Let us start by fixing $[\tau_0]\in\cT_f$, and setting $\tau_n\equiv\sigma_f\tau_0$. Let $E=E(\tau_0)$ be given by \thmref{pil-bound}. 
Further, let us consider the first return map 
$$\cF:\cup \tl S(j)\to \cup \tl S(j),$$
given by Pilgrim's decomposition.

Each $\tl S(j)$ in the domain of $\cF$ is a sphere with a finite number of punctures $P_j\subset \tl S(j)$,
the map 
$\cF:\tl S(j)\to \tl S(j)$ is a Thurston map (\ref{cycle-spheres}), and the postcritical set $P_{\cF}\cap {\tl S(j)}\subset P_j$.
By construction, the extra punctures $P_j\setminus P_\cF$ correspond to cuts in the Pilgrim's decomposition.

 Let $S^0 = \tl S(i)$ be as in the statement of the Main Theorem.
The {\it period} of $S^0$ is the number $N$ of the punctured spheres in the cycle (\ref{cycle-spheres}).
 By analogy with the previously adopted notation, we will 
denote $\cT_\cF$ the Teichm{\"u}ller space of $S^0$, and  $\cM_\cF$ the moduli space.
We further let $$\cT_\cF\overset{\tl p_\cF}{\longrightarrow} \widetilde{\cM_\cF}\overset{\bar p_\cF}{\longrightarrow}\cM_\cF$$ 
be as in \propref{finite cover}.

Finally, for the Fenchen-Nielsen decomposition (\ref{fnp}) we denote $\pi_{\cT_\cF}$ the forgetful map:
$$\pi_{\cT_\cF}:\cT_{S_{\Gamma_c}}\times \mathbb{H}_1 \times \ldots \times \mathbb{H}_k\longrightarrow\cT_{\cF}.$$
We first prove:
\begin{lem}
\label{invariant structure}
There exists a compact set $K\Subset\widetilde \cM_\cF$ 
 with the following property.
Consider the decomposition of $(S,\tau_n)$ given by
$$\mu_n\equiv\pi_{\cT_{\cF}}\circ \Pi_{\Gamma_c}([\tau_n])\in\cT_\cF.$$
Let $\tl\mu_n\equiv \tl p_\cF(\mu_n)\in\widetilde{\cM_\cF}.$
Then, 
for all $n\in\NN$, we have the projections $\tl \mu_n\in K$.
\end{lem}
\begin{proof}
As before, let $\Gamma_c$ stand for the canonical obstruction of $f$. By \thmref{canonical1}, for every $\eps>0$ there exists 
$n_0$ such that for all $n>n_0$ we have
$$[\tau_n]\in \text{Thin}_\epsilon(S,\Gamma_c).$$
By the Collar Lemma, and \lemref{hypest1}, for every $\eps_1>0$ there exists $n_1\in\NN$ such that for every essential
curve $\gamma\subset S_0(j)$ 
$$||\ell_{\tau_n}([\gamma])-\ell_{\tau_n|S_0(j)}([\gamma])||<\eps_1.$$

By \thmref{pil-bound}, for every simple closed  curve $\gamma\subset \cup S_0(j)$ such that:
\begin{itemize}
\item[(a)] $[\gamma]$ is non-trivial in $S$, and
\item[(b)] $[\gamma]\notin\Gamma_c$,
\end{itemize}
we have $\ell_{\tau_n}([\gamma])>E.$
Thus, by Mumford Compactness Theorem, the projections 
$$ p_\cF\circ \pi_{\cT_{\cF}}\circ \Pi_{\Gamma_c}(\tau_n)=\tl p_\cF(\mu_n)$$ lie in
a compact subset $\hat K$ of $\cM_\cF$. Since $\bar p_\cF$ is a finite covering, the set
$$K\equiv (\bar p_\cF)^{-1}(\hat K)\supset\{\tl \mu_n\}$$ is a compact subset of $\widetilde{\cM}_\cF.$

\end{proof}

Let us continue using the notation 
$$\mu_n\equiv\pi_{\cT_{\cF}}\circ \Pi_{\Gamma_c}([\tau_n])\in\cT_\cF.$$
We prove the following:
\begin{prop}
\label{continuity sigma}Let $N$ be the period of $S^0$.
There exists $\eps>0$ such that the following holds. For every $m\in\NN$ there exists 
$k_0=k_0(m)$ such that for all $k\geq k_0$ the distance
$$d_{k,m}\equiv d_{\cT_\cF}([\sigma^m_{\cF}(\mu_{k})],[\pi_{\cT_\cF}\circ \Pi_{\Gamma_c}(\sigma^m_{f^N}(\tau_{k}))])<\eps.$$
\end{prop}
\begin{proof}
Recall that the surface $S^0$ is obtained by capping the surface with holes $S_1(j)\subset S_0(j)$.
Let us {\it combine} the truncated surface 
$$S'\equiv (S\setminus S_0(j),\sigma^m_{f^N}(\tau_{k}))$$
with the surface
$$(S^0,\sigma^m_\cF(\mu_k)).$$
Using \lemref{compn} we can combine the two surfaces into a surface $(S,\lambda_{k,m})$ of the same topological type as $S$, with the only difference
of the complex structure on the component $S_0(j)$, which is now given by $\sigma^m_{\cF}(\mu_k)$. 

 By Minsky's Theorem, for large enough $k$, 
up to a bounded additive error, the Teichm{\"u}ller distance between $(S,\lambda_{k,m})$ and $(S,\sigma^m_{f^N}(\tau_{k}))$ is given by the
distance $d_{k,m}$. 

By the Collar Lemma and 
 \lemref{haiss}, 
$$d_{\cT_f}([(S,\lambda_{k,m})],[(S,\sigma^m_{f^N}(\tau_{k}))])\underset{k\to\infty}{\longrightarrow}0.$$
By Minsky's Theorem \ref{minsky thm},  and \lemref{invariant structure}, 
$d_{k,m}$ is uniformly bounded for large values of $k$.

\end{proof}
Fix $\eps>0$ as in \propref{continuity sigma}, and let $K\Subset \widetilde{\cM_\cF}$ be as in \lemref{invariant structure}.
we can select a subsequence of the natural numbers
$\{ n_j\}$ with the property
$$\tl \mu_{n_j}\to\tl\alpha\in\widetilde{\cM_\cF}.$$
We will endow $\cM_\cF$ and $\widetilde{\cM_\cF}$ with the distance induced by the Teichm{\"u}ller metric on $\cT_\cF$.
Let $\bar\alpha=\bar p_\cF(\tl\alpha)\in\cM_\cF$, and consider a ball $B$ of radius $3\eps$ around $\bar\alpha$. Denote
$$\tl B\equiv (\bar p_\cF)^{-1}(B)\Subset  \widetilde{\cM_\cF}.$$
By \propref{finite cover}, there exists $s\in(0,1)$ such that for every $[\tau]\in \tl p_\cF^{-1}(\tl B)$ the Teichm{\"u}ller norm
of $D_{[\tau]}\sigma_\cF$ is bounded by $s$. 
Hence, there exists $j_0\in\NN$ such that for all $j\geq j_0$ the following holds. 
Let $\tl \zeta\in\tl B$ and let $[\zeta]\in \tl p_\cF^{-1}(\tl\zeta)$ realize the minimum of the Teichm{\"u}ller
distance from $[\mu_{n_j}].$ Let $m\in\NN$ be such that $n_j+m\geq n_{j+k}$. The
distance
$$d_{\cT_\cF}(\sigma_\cF^m([\zeta]),[\mu_{n_j+m}])<s^k d_{\cT_\cF}([\zeta],[\mu_{n_j}]).$$
Let $j$ be such that $\tl \mu_{n_j+k}$ is $\frac{s}{3}$-close to $\tl \alpha$ for all $k\geq 0$. Let us choose $k$ so that for
 $m=n_{j+k}-n_j$ we have $s^m<1/3$. By \propref{continuity sigma}, there exists a branch of 
$\tl\sigma_{\cF^m}$ which maps $\tl B$ compactly inside $B$. Hence, there exists $[\lambda]\in p_\cF^{-1}(B)$ such that
$[\lambda]$ and $\sigma_{\cF^m}([\lambda])$ lie over the same point in $B$. Thus, $\cF^m$ is Thurston equivalent to a rational
map. By Pilgrim's Theorem \ref{canonical1} the map $\cF$ is not obstructed, and the proof of the Main Theorem is completed.

\newpage

\end{document}